\documentclass{article}
\usepackage{amsmath,amsthm,amssymb}
\usepackage{latexsym}
\usepackage[ps,dvips,all,color,line]{xy}
\usepackage{graphicx}
\usepackage{color}
\newcommand{\F}{\mathbb{F}}

 \newcommand{\ds}{\displaystyle}
\theoremstyle{theorem}
\newtheorem{theorem}{Theorem}[section]

\newtheorem{example}[theorem]{Example}

\newtheorem{proposition}[theorem]{Proposition}

\theoremstyle{defi}
\newtheorem{definition}[theorem]{Definition}

\begin{document}\title{\textbf{Unitary Units of The Group Algebra ${\mathbb{F}}_{2^k}Q_{8}$}}

\author{Leo Creedon, Joe Gildea\\
School of Engineering\\
Institute of Technology Sligo\\
Sligo, IRELAND\\
 e-mail: creedon.leo@itsligo.ie, gildeajoe@gmail.com}

%\subjclass[2000]{AMS Subject Classification  20C05  16S84 15A15 15A33}
%\keywords{group ring, units, dihedral}

\date{}

\maketitle

\noindent \textbf{Abstract:} \noindent The structure of the
unitary unit group of the group algebra ${\F}_{2^k} Q_{8}$ is
described as a Hamiltonian group.  \\


\section{Introduction}

\noindent Let $KG$ denote the group ring of the group $G$ over the
field $K$. The homomorphism $\varepsilon : KG \longrightarrow K$
given by $\displaystyle{\varepsilon \left( \sum_{g \in G}a_gg
\right) = \sum_{g \in G}a_g}$ is called the augmentation mapping
of $KG$. The normalized unit group of $KG$ denoted by $V(KG)$
consists of all the invertible elements of $RG$ of augmentation 1.
For further details and background see Polcino Milies and Sehgal
\cite{me5}.

The map $*:KG\longrightarrow KG$ defined by $ \ds{\left( \sum_{g
\in G} a_g g \right)^{*} = \sum_{g \in G} a_g g^{-1}}$ is an
antiautomorphism of $KG$ of order $2$.  An element $v$ of
$V(KG)$ satisfying $v^{-1}=v^{*}$ is called unitary.  We denote by
$V_{*}(KG)$ the subgroup of $V(KG)$ formed by the unitary elements
of $KG$.

Let $char(K)$ be the characteristic of the field $K$.  In
\cite{me2}, A.Bovdi and A. Sz\'{a}kacs construct a basis for
$V_{*}(KG)$ where $char(K) > 2$.  Also A. Bovdi and L. Erdei
\cite{me1} determine the structure of $V_{*}({\F}_{2}G)$ for all
groups of order $8$ and $16$ where ${\F}_{2}$ is the Galois field
of $2$ elements . Additionally in \cite{me3}, V. Bovdi and A.L.
Rosa determine the order of $V_{*}({\F}_{2^k}G)$
 for special cases of $G$.  We establish the structure of $V_{*}({\F}_{2^k}Q_8)$
 to be ${C_2}^{4k-1} \times Q_8$ where $Q_8=\langle x,y \,|\,x^4=1,x^2=y^2,xy=y^{-1}x \rangle$ is the quaternion
 group of order $8$.

\subsection{Background}

\begin{definition}
A circulant matrix over a ring $R$ is a square $n \times n$
matrix, which takes the form
\[ \mbox{circ}(a_1,a_2,\dots,a_n) = \left( \begin{array}{ccccc }

a_1 & a_2 & a_3 & \hdots & a_n\\

a_n & a_1 & a_2 & \hdots &  a_{n-1}\\

a_{n-1} & a_n & a_1 & \hdots & a_{n-2}\\

\vdots & \vdots & \vdots & \ddots &\vdots \\

a_2 & a_3 & a_4 & \hdots & a_1

\end{array}
\right)\] \noindent where $a_i \in R$.
\end{definition}

\noindent For further details on circulant matrices see Davis
\cite{key}.

 Let $\{g_1,g_2, \ldots, g_n\}$ be a fixed listing of the
elements of a group $G$. Then the following matrix:
\[ \left( \begin{array}{ccccc }

{g_1}^{-1}g_1 & {g_1}^{-1}g_2 & {g_1}^{-1}g_3 & \hdots & {g_1}^{-1}g_n\\

{g_2}^{-1}g_1 & {g_2}^{-1}g_2 & {g_2}^{-1}g_3 & \hdots &  {g_2}^{-1}g_n\\

{g_3}^{-1}g_1 & {g_3}^{-1}g_2 & {g_3}^{-1}g_3 & \hdots & {g_3}^{-1}g_n\\

\vdots & \vdots & \vdots & \ddots &\vdots \\

{g_n}^{-1}g_1 & {g_n}^{-1}g_2 &{g_n}^{-1}g_3 & \hdots &
{g_n}^{-1}g_n

\end{array}
\right) \]

\noindent is called the matrix of $G$ (relative to this listing)
and is denoted by $M(G)$. Let $\displaystyle{w = \sum_{i=1}^{n}
\alpha_{g_i}g_i \in RG}$ where $R$ is a ring. Then the following
matrix:

\[  \left( \begin{array}{ccccc }

\alpha_{{g_1}^{-1}g_1} & \alpha_{{g_1}^{-1}g_2} & \alpha_{{g_1}^{-1}g_3} & \hdots & \alpha_{{g_1}^{-1}g_n}\\

\alpha_{{g_2}^{-1}g_1} & \alpha_{{g_2}^{-1}g_2} & \alpha_{{g_2}^{-1}g_3} & \hdots &  \alpha_{{g_2}^{-1}g_n}\\

\alpha_{{g_3}^{-1}g_1} & \alpha_{{g_3}^{-1}g_2} & \alpha_{{g_3}^{-1}g_3} & \hdots & \alpha_{{g_3}^{-1}g_n}\\

\vdots & \vdots & \vdots & \ddots &\vdots \\

\alpha_{{g_n}^{-1}g_1} & \alpha_{{g_n}^{-1}g_2}
&\alpha_{{g_n}^{-1}g_3} & \hdots & \alpha_{{g_n}^{-1}g_n}

\end{array}
\right)\]

\noindent is called the $RG$-matrix of $w$ and is denoted by
$M(RG,w)$. The following theorems can be found in \cite{me4}.

\begin{theorem}
Given a listing of the elements of a group $G$ of order $n$ there
is a ring isomorphism between $RG$ and the $n \times n$
$G$-matrices over $R$.  This ring isomorphism is given
by $\sigma : w \mapsto M(RG,w)$. Suppose $R$ has an identity. Then $w \in  RG$ is a unit if and
only if $\sigma(w)$ is a unit in $M_n(R)$.
\end{theorem}

\begin{example}
Let $Q_8 = \langle x,y \,|\,x^4=1,x^2=y^2,xy=y^{-1}x \rangle$ and
$\kappa =\ds{\sum_{i=0}^{3}a_ix^{i}}+\ds{\sum_{j=0}^{3}b_jx^{j}y}
\in {\F}_{2^k}Q_8$ where $a_i,b_j \in {\F}_{2^k}$. Then
\[ \sigma( \kappa )=
\begin{pmatrix} A & B \\ C & A^T \end{pmatrix} \]
\noindent where $A=\mbox{circ}(a_0,a_1,a_2,a_3)$,
$B=\mbox{circ}(b_0,b_1,b_2,b_3)$ and
$C=\mbox{circ}(b_2,b_1,b_0,b_3)$.
\end{example}

\noindent It is important to note that if $\kappa
=\ds{\sum_{i=0}^{3}a_ix^{i}}+\ds{\sum_{j=0}^{3}b_jx^{j}y} \in
{\F}_{2^k}Q_8$ where $a_i,b_j \in {\F}_{2^k}$, then
$\sigma(\kappa^*) = (\sigma(\kappa))^T$.\\

\noindent The next result can be found in \cite{me3}

\begin{proposition}
Let $K$ be a finite field of characteristic $2$. If $Q_{2^{n+1}}=
\langle a,b\,|\,a^{2^n}=1,\,a^{2^{n-1}}=b^2,\,a^b=a^{-1} \rangle$
is the quaternion group of order $2^{n+1}$, then
\[ |V_{*}(KQ_{2^{n+1}})|=4\cdot |K|^{2^n}. \]
\end{proposition}
\bigskip

\section{The Structure of The Unitary Subgroup of ${\F}_{2^k} Q_{8}$}

\begin{proposition}
$Z(V_{*}({\mathbb{F}}_{2^k}Q_8)) \cong {C_2}^{4k}$ where
$Z(V_{*}({\mathbb{F}}_{2^k}Q_8))$ is the center of
$V_{*}({\mathbb{F}}_{2^k}Q_8)$.
\end{proposition}
\begin{proof}
Let $v =
\ds{\sum_{i=0}^{3}a_ix^{i}}+\ds{\sum_{j=0}^{3}b_jx^{j}y}\in V$
where $V=V({\mathbb{F}}_{2^k}Q_8)$ and $a_i,b_j \in {\F}_{2^k}$.
$C_{V}(x)=\{ v \in V \,|\, xv=vx\}$.  Then
$xv-vx=(b_3-b_1)(y)+(b_0-b_2)xy+(b_1-b_3)x^2y+(b_2-b_0)x^3y$.  If
$\kappa=\ds{\sum_{l=0}^{3}c_lx^{l}}+d_1(y+x^2y)+d_2(xy+x^3y)$
where $\ds{\sum_{l=0}^{3}c_l=1}$ and $d_j \in \F_{2^k}$, then
$\kappa x=x \kappa$. Thus every element of $C_{V}(x)$ has the form
$\ds{\sum_{i=0}^{3}a_i x^{i}}+\gamma_1(y+x^2y)+\gamma_2(xy+x^3y)$
where $\ds{\sum_{i=0}^{3}a_i}=1$ and $\gamma_j \in {\F}_{2^k}$.

$Z(V)$ is contained in $C_{V}(x)$. Therefore $Z(V)=\{\alpha \in
C_{V}(x) \,|\, \alpha v = v\alpha \;\mbox{for all}\;v \in V\}$.
Let $\alpha = \ds{\sum_{i=0}^{3}a_i}x^{i}+b_1(y+x^2y)+b_2(xy+x^3y)
\in C_{V}(x)$ and $v= \ds{\sum_{l=0}^{3}c_lx^{l}}+
\ds{\sum_{m=0}^{3}d_m x^{m}y} \in V$ where $a_i,b_j,c_l,d_m \in
{\F}_{2^k}$. Then
\begin{align*}
\sigma(\alpha)\sigma(v)-\sigma(v)\sigma(\alpha)&=\begin{pmatrix}
 A & B \\ B & A^{T}
\end{pmatrix}\begin{pmatrix}
 C & D \\ E & C^{T}
\end{pmatrix}-\begin{pmatrix}
 C & D \\ E & C^{T}
\end{pmatrix}\begin{pmatrix}
 A & B \\ B & A^{T}
\end{pmatrix}\\
&=\begin{pmatrix} 0 & D(A-A^T) \\ E(A^T-A) & 0 \end{pmatrix}
\end{align*}

\noindent where $A=\mbox{circ}(a_0,a_1,a_2,a_3)$,
$B=\mbox{circ}(b_0,b_1,b_0,b_1)$,
$C=\mbox{circ}(c_0,c_1,c_2,c_3)$, $D=\mbox{circ}(d_0,d_1,d_2,d_3)$
and $E=\mbox{circ}(d_2,d_1,d_0,d_3)$, since circulant matrices commute and $B(E-D)=0=B(C-C^T)$.

\noindent  Therefore
$\sigma(\alpha)\sigma(v)-\sigma(v)\sigma(\alpha)=0$ if $D(A-A^T)=0$ and $E(A^T-A)=0$. It can be
shown that $D(A-A^T)=0$ and $E(A^T-A)=0$ iff $a_1=a_3$. Thus every
element of the $Z(V)$ has the form
$1+r+sx+rx^2+sx^3+ty+uxy+tx^2y+ux^3y$ where $r,s,t,u
\in{\F}_{2^k}$.  It can easily be shown that $Z(V)$ has exponent
$2$.

Now $\alpha^*=\alpha^{-1} \Longleftrightarrow
\sigma(\alpha^*)=\sigma(\alpha^{-1}) \Longleftrightarrow
(\sigma(\alpha))^T=\sigma(\alpha)^{-1} \Longleftrightarrow
\sigma(\alpha)(\sigma(\alpha))^T=I$.  Let $\alpha
=1+r+sx+rx^2+sx^3+ty+uxy+tx^2y+ux^3y \in Z(V)$ where $r,s,t,u
\in{\F}_{2^k}$.  Then
\[
\sigma(\alpha)(\sigma(\alpha))^T  = \begin{pmatrix}
 A & B \\ B & A
\end{pmatrix} \begin{pmatrix}
 A & B \\ B & A
\end{pmatrix}^T=\begin{pmatrix}
 A^2+B^2 & 0 \\ 0 & A^2+B^2
\end{pmatrix}=\begin{pmatrix}I&0\\0&I \end{pmatrix}\]
\noindent where $A=\mbox{circ}(1+r,s,r,s)$,
$B=\mbox{circ}(t,u,t,u)$. Therefore $Z(V)
\subset V_{*}({\mathbb{F}}_{2^k}Q_8)$.

Thus $Z(V_{*}({\mathbb{F}}_{2^k}Q_8))=Z(V)$ and
$Z(V_{*}({\mathbb{F}}_{2^k}Q_8)) \cong {C_2}^{4k}$.
\end{proof}

\noindent We can now construct the following subgroup lattice of
$V_{*}({\mathbb{F}}_{2^k}Q_8):$\\

\xy (8,0): (-8,0)*+{},(0,1)*+{ V_{*}({\mathbb{F}}_{2^k}Q_8)
Q_8}="b",
(1,-1)*+{Q_8}="c",(-1,0)*+{Z(V_{*}({\mathbb{F}}_{2^k}Q_8))}="d",
(0,-2)*+{Z(V_{*}({\mathbb{F}}_{2^k}Q_8)) \cap Q_8  = \{1,x^2\}
}="e", (0,-3)*+{1}="f",\ar @{=} "b";"d",\ar @{=} "b";"c",\ar @{=}
"d";"e",\ar @{=} "c";"e",\ar @{=} "e";"f",
\endxy
\bigskip

\begin{proposition}
 $Z(V_{*}({\mathbb{F}}_{2^k}Q_8)). Q_8 = V_{*}({\mathbb{F}}_{2^k}Q_8)$.
\end{proposition}
\begin{proof} By
the second isomorphism theorem
$Z(V_{*}({\mathbb{F}}_{2^k}Q_8)).Q_8/Z(V_{*}({\mathbb{F}}_{2^k}Q_8))
$ $\cong Q_8/Z(V_{*}({\mathbb{F}}_{2^k}Q_8)) \cap Q_8$.
$|Q_8/Z(V_{*}({\mathbb{F}}_{2^k}Q_8)) \cap Q_8| =
\ds{\frac{8}{2}}= 4$. Therefore
$|Z(V_{*}({\mathbb{F}}_{2^k}Q_8)).Q_8| = 4.2^{4k}=2^{4k+2}$.
Therefore $Z(V_{*}({\mathbb{F}}_{2^k}Q_8)). Q_8 =
V_{*}({\mathbb{F}}_{2^k}Q_8)$.
\end{proof}

\begin{theorem}
$V_{*}({\mathbb{F}}_{2^k}Q_8) \cong {C_{2}}^{4k-1} \times Q_8$.
\end{theorem}
\begin{proof}
$Z(V_{*}({\mathbb{F}}_{2^k}Q_8)) \cong {C_2}^{4k}$ is a vector
space over ${\F}_2$ of dimension $4k$.  Let
$\{x_1,x_2,\ldots,x_{4k}=x^2\}$ be a basis for this vector space.
Therefore $Z(V_{*}({\mathbb{F}}_{2^k}Q_8)) = \langle
x_1,x_2,\ldots,x_{4k} \rangle$.  Let $G=\langle
x_1,x_2,\ldots,x_{4k-1} \rangle$, then $G \cong {C_2}^{4k-1}$ and
$Z(V_{*}({\mathbb{F}}_{2^k}Q_8)) \cong G \times \langle x_{4k}
\rangle \cong G \times \langle x^2 \rangle$. Now $G \cap Q_8 =
\{1\}$ and $V_{*}({\mathbb{F}}_{2^k}Q_8)=G.Q_8$,  therefore
$V_{*}({\mathbb{F}}_{2^k}Q_8) \cong G \rtimes Q_8 \cong G \times
Q_8$ since $G <Z(V_{*}({\mathbb{F}}_{2^k}Q_8))$.  Thus
$V_{*}({\mathbb{F}}_{2^k}Q_8) \cong {C_{2}}^{4k-1} \times Q_8$.
\end{proof}

\noindent The authors wish to acknowledge some useful comments
made by Mazi Shirvani.

\end{document}